\documentclass[12pt,reqno]{amsart}
\usepackage{amssymb}
\usepackage{txfonts}
\usepackage{amsmath}
\usepackage{multirow}

 \newtheorem{thm}{Theorem}[section]
 \newtheorem{cor}[thm]{Corollary}
 \newtheorem{lem}[thm]{Lemma}

 \theoremstyle{definition}
 \newtheorem{defn}[thm]{Definition}
 \theoremstyle{remark}
 \newtheorem{rem}[thm]{Remark}
  
 \numberwithin{equation}{section}

 \newcommand{\tabincell}[2]{\begin{tabular}{@{}#1@{}}#2\end{tabular}}

\begin{document}

\title
{On the Type IIb solutions to mean curvature flow }

\author{Liang Cheng}

\dedicatory{}
\date{}



\thanks{Liang Cheng's Research partially supported by China Scholarship Council no.201906775001, self-determined research funds of CCNU from the colleges' basic research and operation of MOE CCNU19QN075
 }

\address{School of Mathematics and Statistics $\&$ Hubei Key Laboratory of Mathematical Sciences, Central  China Normal University, Wuhan, 430079, P.R.China
}

\email{chengliang@mail.ccnu.edu.cn}

 \maketitle

\begin{abstract}
In this paper we study the Type IIb mean curvature flow.
We first prove that if the convex entire graph
$(y,u(|y|))$   over $\mathbb{R}^n$, $n\geq 2$, satisfying there exist positive constants $\epsilon$, $c$ and $N$ such that
$
  u'(r)\geq c r^{\epsilon}
$
for $r\geq N$,
 the longtime solution to mean curvature flow with initial data $(y,u(|y|))$ must be Type IIb. 
We also study the asymptotic behavior of Type IIb mean curvature flow and show that the limit of suitable rescaling sequence for mean-convex Type IIb mean curvature flow satisfying $\delta$-Andrews' noncollapsing condition is translating soliton. 
\end{abstract}

\section{introduction}

Let $x_0:M^n\to \mathbb{R}^{n+1}$ be a complete immersed hypersurface. Consider the mean curvature
flow
\begin{equation}\label{mcf}
\frac{\partial x}{\partial t}=\vec{\mathbf{H}},
\end{equation}
with the initial data $x_0$, where $\vec{\mathbf{H}}=-H\nu$ is the mean curvature vector and $\nu$ is the outer unit
normal vector. Denote the images $x(M^n,t)=M_t$. The mean curvature flow always blows up at finite time on closed hypersurfaces. However,
the mean curvature flow for noncompact hypersurfaces may have a smooth solution which exists for all time $t>0$,
for which we call it the longtime solution.
Ecker and Huisken \cite{EH1} showed that the mean curvature flow on locally Lipschitz continuous entire graph over
$\mathbb{R}^{n}$ has a longtime solution. Notice that if $M_0$ can be written as an entire graph, then the parabolic system (\ref{mcf}) up to tangential diffeomorphisms is equivalent to
the following quasilinear equation
\begin{equation}\label{mcf_graph}
\frac{d u}{d t}=\sqrt{1+|Du|^2}div(\frac{Du}{\sqrt{1+|Du|^2}})
\end{equation}
(see \cite{EH1}), where $u$ is the graph representation for $M_t$.

Analogous to which was introduced by Hamilton \cite{Ha1} for Ricci flow, one can classify the
longtime solutions to the mean curvature flow into the following two types:
\begin{defn}\label{def_Type_III}
 The longtime solution  to  the mean curvature flow $M_t$
  is called

 (1) Type IIb   if $\sup\limits_{M^n\times
(0,+\infty)}t|A|^2=\infty$, \noindent

(2)
 Type III   if $\sup\limits_{M^n\times
(0,+\infty)}t|A|^2<\infty$,\\
where $A(\cdot,t)$ is the second fundamental form of $M_t$.
\end{defn}
\noindent
In this paper we study the Type IIb mean curvature flow.
The nontrivial examples for the Type IIb
mean curvature flow are gotten in this paper. We also study the asymptotic behavior for the Type IIb mean curvature flow.

Recall the singularity formation of the mean curvature flow   on closed hypersurfaces at the first singular time was described by Huisken \cite{H1}  as follows:
 The solution to  mean curvature flow $M_t$  on closed hypersurfaces which blows up at first finite time $T$ is called

(1) Type I  if $\sup\limits_{M^n\times
[0,T)}(T-t)|A|^2<\infty$, \noindent

(2) Type II   if $\sup\limits_{M^n\times
[0,T)}(T-t)|A|^2=\infty$.

\noindent Using a monotonicity formula, Huisken \cite{H1} showed that
Type I singularities of mean curvature flow are smooth asymptotically like
self-shinkers.
For the compact Type II mean curvature flow, choose the blowup sequence $(p_j,t_j)$  such that $t_j\in [0,T-\frac{1}{j}]$, $p_j\in M^n$, and
\begin{equation}
  H^2(p_j,t_j)(T-\frac{1}{j}-t_j)=\max\limits_{M^n\times [0,T-\frac{1}{j}]}H^2(p,t)(T-\frac{1}{j}-t)
\end{equation}
Let
$L_j = |H(p_j, t_j)|$.
Consider the following rescaled mean curvature flows
\begin{equation}\label{typeii}
M^j_t= L_j(M_{t_j+L_j^{-2}t}-x(P_j,t_j)),
\end{equation}
for $t\in[\alpha_j, \Omega_j]$, where
$\alpha_j = -t_j L^2_j \to -\infty$ and
$\Omega_j=(T-t_j-\frac{1}{j})L^2_j\to+\infty$.
For each rescaled flow (\ref{typeii}), $0 \in M^j_0$ and $|H_j|$ achieves the maximum value 1 at $t=0$.
By employing a Harnack inequality,
Hamilton \cite{Hamilton} showed that any strictly convex eternal solution to the mean curvature
flow where the mean curvature assumes its maximum value at a point in space-time must
be a translating soliton.
Huisken and Sinestrari (\cite{HS1} \cite{HS2}) proved blowup  sequence of the compact mean curvature flow with positive mean curvature subconverges to a weakly convex  limit
splitting as $\mathbb{R}^{n-k}\times \Sigma^k$, where $\Sigma^k$ is strictly convex.
Their results implies that
rescaled sequence (\ref{typeii}) subconverges to a
translating soliton if $M_0$ is mean-convex.

Similar to the Type II mean curvature flow, one can also choose suitable rescaled sequence such that the limit is an eternal solution with mean curvature achieving 
the maximum value  in the space-time.
If $M_t$ is the Type IIb mean curvature flow with bounded second fundamental form at each time slice on noncompact hypersurface,
as which was introduced by Hamilton \cite{Ha1} for Ricci flow, one can choose $j\to +\infty$, and pick $P_j$ and $t_j$ such that
\begin{equation}
t_j (j-t_j) H^2(P_j,t_j)\geq \gamma_j \sup\limits_{M^n\times[0,j]}t (j-t) H^2(P,t),
\end{equation}
where $\gamma_j\nearrow1$. Let $L_j=|H|(P_j,t_j)$. Consider the following the rescaled mean curvature flows
\begin{equation}\label{TypeIIb_rescaling}
M^j_t=L_j (M_ {t_j+L_j^{-2}t}-x(P_j,t_j)),
\end{equation}
for $t\in [\alpha_j,\Omega_j]$, where
$\alpha_j=-t_jL_j^2$ and $\Omega_j=(j-t_j)L_j^2$.
Then
$$H^2_j(\cdot, t)\leq \gamma_j^{-1}\frac{\alpha_j}{\alpha_j-t}\frac{\Omega_j}{\Omega_j-t},$$
for $t\in [\alpha_j,\Omega_j]$. Now
$$
\frac{1}{-\alpha_j^{-1}+\Omega_j^{-1}}\geq \gamma_j j^{-1}\sup\limits_{M^n\times [0,j]}(t(j-t)H^2(x,t))
\geq \frac{\gamma_j}{2}\sup\limits_{M^n\times[0,\frac{j}{2}]}tH^2(x,t)\to+\infty,
$$
Hence $\alpha_j\to -\infty$ and $\Omega_j\to +\infty$. If the Type IIb mean curvature flow is convex, one get the limit $M_{\infty}$ of the rescaled sequence (\ref{TypeIIb_rescaling})
  is an  eternal solution splitting as $\mathbb{R}^{n-k}\times \Sigma^k$  with its mean curvature achieving the maximum value $1$ in the space-time, where $\Sigma^k$ is strictly convex.
This implies that $M_{\infty}$ is a translating soliton by Hamilton's Harnack inequality (\cite{Hamilton}).

For the Type III mean curvature flow, rescaling the mean curvature flow  as
\begin{equation}\label{scaling}
\tilde{x}(\cdot,s)=\frac{1}{\sqrt{2t+1}}x(\cdot,t),
\end{equation}
where $s$ is given by $s=\frac{1}{2}\log(2t+1)$. The normalized mean curvature flow  then becomes
\begin{equation}\label{normalized_mcf}
\frac{\partial \tilde{x}}{\partial s}=\vec{\tilde{\mathbf{H}}}-\tilde{x}.
\end{equation}
Note that Type III condition implies $\sup\limits_{M^n\times
(0,+\infty)}|\tilde{A}|<\infty$. If the Type \  III mean curvature flow is convex, we can use  Hamilton's Harnack inequality to get the limit $M_{\infty}$ is a non-flat self-expander splitting as $\mathbb{R}^{n-k}\times \Sigma^k$, where $\Sigma^k$ is strictly convex (see Corollary \ref{self_expander} in the appendix). We remark that
 a counter-example in \cite{CN}(see Example 3.4 in \cite{CN}) shows that the rescaled sequence (\ref{scaling}) can not converge to the self-expander if only assuming the Type III mean curvature flow is mean-convex.
Recently
the author and Sesum \cite{CN} also introduced  monotonicity formulas related
to self-expanders and showed the normalized  flow (\ref{normalized_mcf}) for Type III mean curvature flow subconverges to the self-expander under certain conditions.

Typical examples of the Type III mean curvature flow are evolving entire graphs satisfying
the following condition\begin{equation}\label{linear_growth}
\upsilon:=\langle \nu,\omega\rangle^{-1}\leq c,
\end{equation}
which  in particular implies the entire graphs having the bounded gradient, where $\nu$ are the unit normal vectors of the graph and $\omega$ is a fixed unit vector such that $\langle \nu,\omega\rangle>0$.
 Ecker and Huisken showed that the mean curvature flow on entire graphs satisfying the condition (\ref{linear_growth}) is  Type III (Corollary 4.4 in \cite{EH}).
Moreover, Ecker and Huisken \cite{EH1} also proved  that if the entire graph satisfies
condition (\ref{linear_growth})
 and the estimate
\begin{equation}\label{growth_condition}
\langle x_0,\nu \rangle^2 \leq c(1+|x_0|^2)^{1-\delta}
\end{equation}
at time $t = 0$, where  $c<\infty$ and $\delta>0$, then
 the solution to the normalized mean curvature flow (\ref{normalized_mcf})
with initial data $x_0$
converges as $s\to \infty$ to a self-expander.

In contrast to Type III mean curvature flow, much less examples are known about the Type IIb mean curvature flow except non-flat translating solitons.
The following theorem  lead us to get the nontrivial examples for the Type IIb
mean curvature flow. Compare to the result by Ecker and Huisken that 
the mean curvature flow of entire graph which has the bounded gradient is Type III (Corollary 4.4 in \cite{EH}), we have the following theorem.

\begin{thm}\label{application_2}
 Let $(y,u(|y|))$ be the convex entire graph over $\mathbb{R}^n$, $n\geq 2$. Suppose that
 there exist positive constants $\epsilon$, $c$ and $N$ such that
 for $r\geq N$
 \begin{equation}\label{qqqq}
  u'(r)\geq c r^{\epsilon}.
 \end{equation}
 Then the longtime solution to mean curvature flow with initial data $(y,u(|y|))$ must be Type IIb.
\end{thm}
\begin{rem}
	(1) Theorem \ref{application_2} is a corollary of more generalized Theorem \ref{main1} in section 3, which is not only valid for the entire graphs.
	
	(2) Ecker and Huisken's result \cite{EH1} showed that
	the mean curvature flow for the entire graph which has the bounded gradient must be Type III (Corollary 4.4 in \cite{EH}). This implies that Theorem \ref{application_2} does not hold if $\epsilon\leq 0$.

\end{rem}	
  
 Applying  Ecker and Huisken's result in \cite{EH1}  and
 summarizing the results in this paper, we can prove  the following interesting table in section \ref{table}.
\begin{table}[ht]
\caption{Mean curvature flow for convex entire graph $(y,u(|y|))$ over $\mathbb{R}^{n}$ with $u'(r)=O(r^{\alpha-1})$, $n\geq2$.
 }
   \begin{tabular}{|c|c|c|c|}
     \hline
      &  \tabincell{c}{ When $t\to+\infty$,\\ for any $ p\in M^n$}   & Singularity Type & Asymptotic behavior \\
     \hline
     $\alpha\leq1$ & $|A|(p,t)\to 0 $ & Type III & \tabincell{c}{ rescaled sequence (\ref{scaling})\\ subconverges to \\a self-expander} \\
     \hline
     $1<\alpha<2$ &$|A|(p,t)\to 0 $& \multirow{3}{*}{Type IIb} & \multirow{3}{*}{\tabincell{c}{ rescaled sequence (\ref{TypeIIb_rescaling})\\ subconverges to\\ a translating soliton}} \\
     \cline{1-2}
     $\alpha=2$ & $0<c_p\leq|A|(p,t)\leq C$ &  &  \\
     \cline{1-2}
     $\alpha>2$ & $|A|(p,t)\to +\infty$ &  &  \\
     \hline
   \end{tabular}
\end{table}

\newpage

Finally, we study the asymtotic behavior of the Type IIb mean curvature flow.
We show that Huisken and Sinestrari's result for Type II also holds for 
$\delta$-Andrews' noncollapsed Type IIb mean curvature flow.
\begin{thm}\label{andrew_application2}
If the Type IIb mean curvature flow is mean-convex and satisfies the $\delta$-Andrews' noncollapsing condition (see Definition \ref{andrew_condition}) at $t=0$, then  limit $M_{\infty}$ of the rescaled sequence (\ref{TypeIIb_rescaling})
  is a weakly convex and hence $M_{\infty}$ is the tanslating soliton
splitting as $\mathbb{R}^{n-k}\times \Sigma^k$, where $\Sigma^k$ is strictly convex.
\end{thm}

The structure of this paper is as follows. In section 2 we give proof of Theorem \ref{application_2}. 
In section \ref{table} we give the proof of  Table 1.
In section \ref{sec_3} we extend Andrews'proof of noncollapsing theorem to noncompact case (Theorem \ref{BenAndrew}) and give the proof of Theorem \ref{andrew_application2}. In this section we also  prove an estimate (\ref{bound_3}) which is used in the proof of Theorem \ref{main1}.
 In Appendix 
we show that convex normalized Type III mean curvature flow  (\ref{normalized_mcf}) 
subconverges a non-flat self-expander which is mentioned in the introduction.

\section{Proof of Theorem \ref{application_2}}

In this section we obtain an generalized result (Theorem \ref{main1})  to show that the longtime solution to mean curvature flow satisfying certain conditions would be Type IIb. Theorem 1.2 is a corollary of Theorem \ref{main1}.

We first recall the definition of $\delta$-Andrews' noncollapsing condition.
\begin{defn}\cite{A}\cite{SW}\cite{Hasl}\label{andrew_condition}
  ($\delta$-Andrews' noncollapsing condition) If  $M$ is a smooth, complete,
mean-convex embedded hypersurface (possibly noncompact) with $M=\partial K$, then $M$
satisfies the $\delta$-Andrews' noncollapsing condition for $\delta>0$ if for every $p\in M$
 there are closed balls
$\bar{B}_{Int}\subseteq K$ and $\bar{B}_ {Ext}\subseteq \mathbb{R}^{n+1}\backslash Int(K)$ of radius at least $\frac{\delta}{H(p)}$ that are tangent to $M$ at $p$ from the interior and exterior of $M$ respectively.
\end{defn}
\begin{rem}
Compare to the definitions of $\delta$-Andrews' noncollapsing condition used in \cite{A}\cite{SW}\cite{Hasl}, the hypersurfaces satisfying the $\delta$-Andrews' noncollapsing condition in Definition \ref{andrew_condition} could be noncompact. Clearly, the closed mean-convex embedded
hypersurfaces always satisfying the $\delta$-Andrews' noncollapsing condition
for some $\delta>0$. Howerver, there are some  noncompact  mean-convex embedded
hypersurfaces, for example the grim reaper, do not satisfy the $\delta$-Andrews' noncollapsing condition.
\end{rem}

The main theorem of this section is the following
\begin{thm}\label{main1}
(A)
Let $M_t$ be a solution to the mean curvature flow  for immersed noncompact hypersurface in $\mathbb{R}^{n+1}$.
 Suppose that there exist a fixed vector $\omega$ and constants $C_1$, $C_2$  such that
\begin{equation}\label{pinching_condition_1}
C_1H\leq  W\leq C_2H,
\end{equation}
at $t=0$, where $ W=\langle \nu,\omega\rangle$ and $\omega$ is a constant vector field, and
\begin{align}\label{condtion2}
&\text{$M_0$ can be contained in the half-plane $\mathbb{R}^{n+1}_+$ with its boundary
 $\partial\mathbb{R}^{n+1}_+$ not }\nonumber\\
&\text{parallel     to   $\omega$.}
\end{align}
Then the mean curvature flow $M_t$ can not be Type III. In addition, if $M_0$ is an entire graph, then the longtime solution to mean curvature flow for $M_0$ must be Type IIb.\\
(B) If $M_0$ is mean-convex and  satisfies the $\delta$-Andrews' noncollapsing condition, the condition (\ref{condtion2}) and
\begin{equation}\label{pinching_condition_2}
  |W|\leq C(1+|x|^2)^{\frac{1-\epsilon}{2}}H,
\end{equation}
at $t=0$ for some $\epsilon>0$, then the conclusions of (A) still hold.
\end{thm}
\begin{rem}
(1) Let $M_0$ be the hyperplane in $\mathbb{R}^{n+1}$. We choose $\omega$ be a fixed vector which is parallel to the hyperplane, and hence $H=W\equiv 0$. Then $M_0$
 satisfies the condition (\ref{pinching_condition_1}) rather than condition (\ref{condtion2}), and clearly the mean curvature flow with initial data $M_0$ is not Type IIb.
  This shows that the  condition (\ref{condtion2}) in Theorem \ref{main1} can not be removed.

(2) 
Notice that he grim reaper satisfies the conditions of Theorem \ref{main1} (A) rather than the conditions of  Theorem \ref{main1} (B).

	(3) Again by Ecker and Huisken's result \cite{EH1} that
	the mean curvature flow for the entire graph which has the bounded gradient must be Type III (Corollary 4.4 in \cite{EH}) we know that Theorem \ref{application_2} (B) does not hold if $\epsilon\leq 0$ in (\ref{pinching_condition_2})(see the calculations in the proof of  Theorem \ref{application_2}).
  \end{rem}

\textbf{Proof of Theorem \ref{main1}.} (A)
We argue by contradiction.
Assume that $M_t=x(M^n,t)$ is the Type III solution to the mean curvature flow  with $M_0$ satisfying conditions (\ref{pinching_condition_1}) and (\ref{condtion2}).
Let $\tilde{x}(\cdot,s)$ be its corresponding normalized mean curvature flow (\ref{normalized_mcf}). Denote $\tilde{M}_s=\tilde{x}(M^n,s)$. 
Since 
\begin{equation}\label{preserved}
(\frac{\partial}{\partial t}-\Delta)(H-CW)=|A|^2(H-CW)
\end{equation}
holds for any constant $C$ and by
the maximum principle (Theorem 4.3 in \cite{EH1})
 that $C_1H\leq  W\leq C_2H$ holds for all $t\geq 0$.
By (\ref{scaling}) and Type III condition, we have $\sup\limits_{M^n\times
[0,\infty)}|\tilde{A}|^2=C<\infty$ , $\tilde{W}=W$ and $\tilde{H}=\sqrt{2t+1}H$, where $\tilde{W}=\langle \tilde{\nu},\omega\rangle$ and $\tilde{A}$, $\tilde{H}$ are the second fundamental form and mean curvature for $\tilde{x}(\cdot,s)$.
Hence
\begin{equation}
  -\frac{nCC'}{\sqrt{2t+1}}\leq |\tilde{W}| \leq \frac{nCC'}{\sqrt{2t+1}},
\end{equation}
where $C'=\max\{|C_1|,|C_2|\}$, which implies that
\begin{equation}\label{Wtozero}
|\tilde{W}|\to 0
\end{equation}
as $s\to+\infty$.

We calculate that
\begin{equation}\label{eq_31}
\frac{\partial }{\partial s} |\tilde{x}|^2=2\langle\vec{\tilde{\mathbf{H}}},\tilde{x}\rangle-2|\tilde{x}|^2.
\end{equation}
It follows that
$$|\tilde{x}|(p,s)\leq e^{-s}|\tilde{x}_{0}|(p)+nC (1 - e^{-s}),$$
for any $p\in M^n$.
Hence for any fixed point $p\in M^n$,
\begin{equation}\label{bound}
|\tilde{x}|(p,s)\leq nC+1
\end{equation}
for $s$  sufficiently large.

We next use the technique in \cite{CH} to get an intrinsic limit.
Since  $|\tilde{A}|\leq \sqrt{C}$ for $[0,+\infty)$, the injectivity radius of $(M^n,\tilde{g}(s))$ at $p$ has the positive lower bound only depending on $C$.
 We also have $|\tilde{\nabla}^m \tilde{A}|\leq C_m$ on $[0,+\infty)$ by the standard estimates by Ecker and Huisken \cite{EH1}.
 It follows that there exists a sequence $s_i\to+\infty$ such that
$(M^n,\tilde{g}(s_i),p)$  converges to a complete manifold $(M^n_{\infty},\tilde{g}_{\infty},p_{\infty})$ in $C^{\infty}$ pointed Gromov-Hausdorff sense. That is, for any $r>0$ and $i$, there exist embeddings $\phi_i:B_{\tilde{g}_{\infty}}(p_{\infty},r)\to M^n$ such that
$\phi_i(p_{\infty})=p$ and $\phi_i^*\tilde{g}_i$ converges smoothly to $\tilde{g}_{\infty}$ on $B_{\tilde{g}_{\infty}}(p_{\infty},r)$, where $B_{\tilde{g}_{\infty}}(p_{\infty},r)$ is the intrinsic ball on $(M^n_{\infty},\tilde{g}_{\infty},p_{\infty})$.
Since
\begin{equation}\label{eq_xxxxx}
(\phi_i^*\tilde{g}_i)_{kl}=\partial_k \tilde{x}_i\circ \phi_i\cdot \partial_l \tilde{x}_i\circ \phi_i,
 \end{equation}
 the first derivatives of $\tilde{x}_i\circ \phi_i$ are uniformly bounded on $B_{\tilde{g}_{\infty}}(p_{\infty},r)$. Moreover, by the Gauss-Weingarten relations and  $|\tilde{\nabla}^m \tilde{A}|\leq C_m$,
\begin{align*}
&\partial_k\partial_l F=\Gamma^q_{kl}\partial_q F-h_{kl} \nu,\\
&\partial_k\nu=h_{kl}\tilde{g}^{lq} \partial_q F,
\end{align*}
here $F=\tilde{x}_i\circ \phi_i$ and $g=\phi_i^*\tilde{g}_i$ in our case,
we have all derivatives of $\tilde{x}_i\circ \phi_i$ are uniformly bounded on $B_{\tilde{g}_{\infty}}(p_{\infty},r)$.
Notice that $\tilde{x}_i\circ \phi_i:B_{\tilde{g}_{\infty}}(p_{\infty},r)\to \mathbb{R}^{n+1}$ and $\tilde{x}_i\circ\phi_i$ is uniformly bounded at $p_{\infty}$ by (\ref{bound}). It follows that $\tilde{x}_i\circ \phi_i$ subconverges smoothly to map $\tilde{x}_{\infty}:B_{\tilde{g}_{\infty}}(p_{\infty},r)\to \mathbb{R}^{n+1}$. Let $i\to \infty$ in (\ref{eq_xxxxx}), we get
$(\tilde{g}_{\infty})_{kl}=\partial_k \tilde{x}_{\infty}\cdot \partial_l \tilde{x}_{\infty}$.  Since
$(\tilde{g}_{\infty})_{kl}$ is positive definite matrix, $\tilde{x}_{\infty}$ is an immersion on $B_{\tilde{g}_{\infty}}(p_{\infty},r)$. By the standard diagonal argument and taking $r=r_l\to \infty$, we get the complete immersion $\tilde{x}_{\infty}:M^n_{\infty}\to \mathbb{R}^{n+1}$.

   By (\ref{Wtozero}) we conclude that $\tilde{W}_{\infty}=0$, which implies $\tilde{x}_{\infty}$ is complete cylindrical hypersurface, i.e. the hypersurface splitting as $\Sigma^{n-1}\times l$  with $l$ is a straight line parallel to $\omega$. Hence $\tilde{x}_{\infty}(M^n_{\infty})$ must cross over the plane $x_{n+1}=0$. However the plane $x_{n+1}=0$
 is steady under the normalized mean curvature flow (\ref{scaling}), which implies $\tilde{M}_{s}$ is contained in $\mathbb{R}^{n+1}_+$ for all $s\geq 0$. It follows that
 $\tilde{x}_{\infty}(M^n_{\infty})$ is contained in $\mathbb{R}^{n+1}_+$. Then we get a contradiction.

(B)
We argue by contradiction.
Assume that $M_t=x(M^n,t)$ is the Type III solution to the mean curvature flow  with $M_0$ satisfying conditions $\delta$-Andrews' noncollapsing condition, (\ref{condtion2}) and (\ref{pinching_condition_2}).
 Under the  normalized mean curvature flow (\ref{normalized_mcf}), we have
 $$(\frac{\partial}{\partial s}-\tilde{\Delta})\tilde{W}=|\tilde{A}|^2\tilde{W},$$
 and
 $$(\frac{\partial}{\partial s}-\tilde{\Delta})\tilde{H}=|\tilde{A}|^2\tilde{H}+\tilde{H}.$$
 It follows from Corollary \ref{nabla1} in section \ref{sec_3} that
 \begin{equation}\label{bound_3}
 \frac{|\tilde{\nabla}^l \tilde{A}|}{\tilde{H}^{l+1}}\leq C_l(\delta).
 \end{equation}
 Hence we
 have
 \begin{equation}\label{bound_1}
 \frac{|\tilde{\nabla} \tilde{H}|}{\tilde{H}}\leq m_1,
 \end{equation}
 and
 \begin{equation}\label{bound_2}
 \frac{|\frac{\partial}{\partial s }\tilde{H}|}{\tilde{H}}\leq m_2,
 \end{equation}
 where $m_1$ and $m_2$ are positive constants depending on $\delta$ and $\sup\limits_{M^n\times
 [0,\infty)}|\tilde{A}|$.
 Then
 \begin{align*}
 \left(\frac{\partial }{\partial s}-\tilde{\Delta}\right)\,\frac{\tilde{W}^2}{\tilde{H}^2}&=\frac{(\frac{\partial }{\partial s}-\tilde{\Delta})\tilde{W}^2}{\tilde{H}^2}
 -\frac{\tilde{W}^2(\frac{\partial }{\partial s}-\tilde{\Delta})|\tilde{H}|^2}{\tilde{H}^4}+2\tilde{\nabla} \log \tilde{H}^2\cdot\tilde{\nabla} \frac{\tilde{W}^2}{\tilde{H}^2}\\
 &=-2\frac{\tilde{W}^2}{\tilde{H}^2} -\frac{2|\tilde{\nabla} \tilde{W}|^2}{\tilde{H}^2}
 +\frac{2\tilde{W}^2|\tilde{\nabla} \tilde{H}|^2}{\tilde{H}^4}+4\frac{\tilde{\nabla} \tilde{H}}{\tilde{H}}\cdot\tilde{\nabla} \frac{\tilde{W}^2}{\tilde{H}^2}\\
 &=-2\frac{\tilde{W}^2}{\tilde{H}^2} -\frac{2|\tilde{\nabla} \tilde{W}|^2}{\tilde{H}^2}
 -\frac{2\tilde{W}^2|\tilde{\nabla} \tilde{H}|^2}{\tilde{H}^4}+4\frac{\tilde{\nabla} \tilde{H}}{\tilde{H}}\cdot \frac{\tilde{W}\tilde{\nabla}\tilde{W}}{\tilde{H}^2}\\
 &\ \ +2\frac{\tilde{\nabla} \tilde{H}}{\tilde{H}}\cdot\tilde{\nabla} \frac{\tilde{W}^2}{\tilde{H}^2}\\
 &\leq -2\frac{\tilde{W}^2}{\tilde{H}^2}+2\frac{\tilde{\nabla} \tilde{H}}{\tilde{H}}\cdot\tilde{\nabla} \frac{\tilde{W}^2}{\tilde{H}^2}.
 \end{align*}
  We follow an idea of Ecker and Huisken in \cite{EH}.
 Define
 $\eta_{a }(\tilde{x})=1+a |\tilde{x}|^2$, $\rho(\tilde{x},s)=\eta_{a }^{\epsilon-1}e^{\beta s}$, where $a$ is positive constant to be determined later.
 We calculate that
 $$(\frac{\partial}{\partial s}-\tilde{\Delta})\eta_{a }=-2a  (|\tilde{x}|^2+n),$$
 and hence
 $$(\frac{\partial}{\partial s}-\tilde{\Delta})\rho\leq (\beta+2(1-\epsilon)(a  n+1))\rho.$$
 Moreover, $|\tilde{\nabla}\eta_{a }|^2\leq 4a \eta_{a }$ and $|\tilde{\nabla}\rho|\leq 2a ^{\frac{1}{2}}\rho$.
We compute that
 \begin{align*}
 (\frac{\partial}{\partial s}-\tilde{\Delta})\frac{\tilde{W}^2}{\tilde{H}^2}\rho &=
 \rho(\frac{\partial}{\partial s}-\tilde{\Delta})\frac{\tilde{W}^2}{\tilde{H}^2}+\frac{\tilde{W}^2}{\tilde{H}^2}(\frac{\partial}{\partial s}-\tilde{\Delta})\rho
 -2\langle \tilde{\nabla}\frac{\tilde{W}^2}{\tilde{H}^2},\tilde{\nabla}\rho\rangle\\
 &\leq (\beta+2a  n-2\epsilon)\frac{\tilde{W}^2}{\tilde{H}^2}\rho+2\rho\langle  \tilde{\nabla}\log\tilde {H},\tilde{\nabla}\frac{\tilde{W}^2}{\tilde{H}^2}\rangle-2\langle \tilde{\nabla}\frac{\tilde{W}^2}{\tilde{H}^2},\tilde{\nabla}\rho\rangle\\
 &=(\beta+2a  n-2\epsilon)\frac{\tilde{W}^2}{\tilde{H}^2}\rho+2  (\tilde{\nabla}\log\tilde {H}-\rho^{-1}\tilde{\nabla}\rho)\cdot\tilde{\nabla}(\frac{\tilde{W}^2}{\tilde{H}^2}\rho)\\
 &\ \ \ -2\frac{\tilde{W}^2}{\tilde{H}^2}\langle  \tilde{\nabla}\log\tilde {H},\tilde{\nabla}\rho\rangle
 +2\frac{\tilde{W}^2}{\tilde{H}^2}\rho^{-1}|\tilde{\nabla \rho}|^2\\
 &\leq (\beta+2a  n+4m_1a ^{\frac{1}{2}}+8a -2\epsilon)\frac{\tilde{W}^2}{\tilde{H}^2}\rho\\
 &\ \ \ \ +2  (\tilde{\nabla}\log\tilde {H}-\rho^{-1}\tilde{\nabla}\rho)\cdot\tilde{\nabla}(\frac{\tilde{W}^2}{\tilde{H}^2}\rho).
 \end{align*}
 Taking $a $ and $\beta$ small enough such that $\beta+2a  n+4m_1a ^{\frac{1}{2}}+8a -2\epsilon<0$, we have
 \begin{equation}\label{evolution}
 (\frac{\partial}{\partial s}-\tilde{\Delta})\frac{\tilde{W}^2}{\tilde{H}^2}\rho\leq 2  (\tilde{\nabla}\log\tilde {H}-\rho^{-1}\tilde{\nabla}\rho)\cdot\tilde{\nabla}(\frac{\tilde{W}^2}{\tilde{H}^2}\rho),
 \end{equation}
 with $|\tilde{\nabla}\log\tilde {H}-\rho^{-1}\tilde{\nabla}\rho|\leq m_1+2a ^{\frac{1}{2}}$.
 \begin{align*}
 \frac{\partial}{\partial s}(\frac{\tilde{W}^2}{\tilde{H}^2}\rho) &= 2\frac{\langle\tilde{\nabla}\tilde{H},\omega\rangle\tilde{W}}{\tilde{H}^2}\rho
 -2\frac{\frac{\partial}{\partial s}\tilde{H}}{\tilde{H}}(\frac{\tilde{W}^2}{\tilde{H}^2}\rho)+
 (\epsilon-1)\frac{\frac{\partial}{\partial s}\eta_{a }}{\eta_{a }}(\frac{\tilde{W}^2}{\tilde{H}^2}\rho)+\beta\frac{\tilde{W}^2}{\tilde{H}^2}\rho\\
 &\leq 2C_1(\delta)e^{\beta s}+C'\frac{\tilde{W}^2}{\tilde{H}^2}\rho,
 \end{align*}
 where $C'$ is positive constant depends on $\epsilon, \beta,m_2,a$ and $\sup\limits_{M^n\times
 [0,\infty)}|\tilde{A}|$. So we get that $\sup\limits_{\tilde{M}_s}\frac{\tilde{W}^2}{\tilde{H}^2}\rho$ is finite at each time slice. Applying the maximum
 principle to (\ref{evolution}) (see Corollary 1.1 in \cite{EH}) we have
 \begin{equation}
  \sup\limits_{\tilde{M}_s}\frac{\tilde{W}^2}{\tilde{H}^2}( 1+a |\tilde{x}|^2)^{\epsilon-1}
  \leq e^{-\beta s}\sup\limits_{\tilde{M}_0}\frac{\tilde{W}^2}{\tilde{H}^2}( 1+a |\tilde{x}|^2)^{\epsilon-1}.
 \end{equation}
 Moreover, $\tilde{H}$ is uniformly bounded by the Type III condition, hence $\tilde{W}\to 0$ as $s\to+\infty$ on any compact set.
 Now we can use the same arguments in the proof of Theorem \ref{main1} (A) to get a contradiction.
 $\Box$

Next we give the proof of Theorem \ref{application_2}.\\
\textbf{Proof of Theorem \ref{application_2}.}  
 We choose $\omega=-e_{n+1}$. Define $r=|y|$.  
 By direct calculations,
  \begin{equation}\label{calculation1}
     W=\frac{1}{(1+u'(r)^2)^{\frac{1}{2}}}
  \end{equation}
 and
 \begin{equation}\label{calculation3}
 H=\frac{u''(r)}{(1+u'(r)^2)^{\frac{3}{2}}}+\frac{(n-1)u'(r)}{r(1+u'(r)^2)^{\frac{1}{2}}}.
 \end{equation} 
  Since $(y,u(|y|))$ is convex,  $u''(r)>0$ and $u'(r)>0$.
  By (\ref{calculation1}) and (\ref{calculation3}), we have
   \begin{align*}
     \frac{H}{W}(1+|x|^2)^{\frac{1-\epsilon}{2}}&\geq \frac{(n-1)u'(r)}{r}(1+r^2)^{\frac{1-\epsilon}{2}}\\
      &\geq c(n-1)
   \end{align*}
   for $r\geq N$. If the graph $(y,u(|y|))$ is smooth and covex, then $H$ and $W$ are positive and continuous.  Then
   (\ref{pinching_condition_2}) is satisfied for $r\leq N$.
   By the straightforward computation, we can embed the ball of radius at least $\frac{r\sqrt{1+u'(r)^2}}{u'(r)}$ that are tangent to graph at $|y|=r$ from the interior and exterior of the
   graph respectively. Combining this with (\ref{calculation3}), we
   conclude that $\delta$-Andrews' noncollapsing condition is satisfied.
   Hence the mean curvature flow for the graph $(y,u(|y|))$ over $\mathbb{R}^n$, $n\geq 2$, satisfying (\ref{qqqq})
    must be Type IIb by
   Theorem \ref{main1} (B). 
$\Box$

\section{Proof of Table 1}\label{table}

In this section we give the proof of Table 1.
First we prove the following theorem.
\begin{thm}\label{main3}
Let $M_t$ be the solution to mean curvature flow with initial data $M_0$ is the convex entire graph $(y,u(|y|))$ over $\mathbb{R}^{n}$ with $u'(r)=O(r^{\alpha-1})$, $\alpha> 1$ and $n\geq 2$. If $\alpha> 2$, then the $|A|(p,t)\to+\infty$ as $t\to+\infty$ for any $p\in M^n$.
If $1<\alpha <2$, then $|A|(p,t)\to 0$ as $t\to+\infty$ for any $p\in M$. If $\alpha=2$, then $0<c_p\leq |A|(p,t)\leq C$, where $c_p$ is a positive constant depending on $p$.
\end{thm}
\begin{proof}
In \cite{SW} Altschuler, Steven J.,  L. F. Wu proved that there exists  rotational symmetric convex translating graph $X_0=(y,g_N(|y|))$ which is called bowl soliton satisfying
\begin{equation*}
  \vec{H}=V^{\perp}.
\end{equation*}
with $V=Ne_{n+1}$ and $\lim\limits_{r\to +\infty}\frac{g_N'(r)}{r}=N$ for any $N>0$\footnote{If there  exists  rotational symmetric convex translating graph $X_0=(y,g(|y|))$  satisfying
	$
	\vec{H}=V^{\perp}.
	$
	with $V=e_{n+1}$ and $\lim\limits_{r\to +\infty}\frac{g'(r)}{r}=1$, then by rescaling  there exists  rotational symmetric convex translating graph $X_0=(y,g_N(|y|))$  satisfying
	$
	\vec{H}=V^{\perp}.
	$
	with $V=Ne_{n+1}$ and $\lim\limits_{r\to +\infty}\frac{g_N'(r)}{r}=N$ for any $N>0$. },
 and hence it is "asymptotic to" paraboloid as $r\to+\infty$ (One can also see \cite{C} 
which
gives an explicit calculation of the asymptotics of the bowl soliton and as well as a clearer
construction of the bowl soliton).
 The solution to mean curvature flow
with initial data $X_0$ is translating as $X_t(\phi^*_t(p),t)=X_0(p)+Ne_{n+1}t$.

If $\alpha>2$, then for any large $N$ there exists positive constant $C_N$ such that
$u(|y|)+C_N>g_N(|y|)$. It implies that
$M_0+C_Ne_{n+1}$  is above $X_0$. Now we argue by contradiction.
Assume that there exists $p\in M^n$ such that $|A|(p,t)\leq C_p$. By (\ref{mcf}),
$$|(x(p,t)+C_Ne_{n+1})-(x(p,0)+C_Ne_{n+1})|\leq C_pt.$$
That implies
$$B(x(p,0)+C_Ne_{n+1},C_pt)\cap (M_t+C_Ne_{n+1})\neq \emptyset.$$
However, taking $N=2C_p$, $X_t$ is translating in the $e_{n+1}$ direction with velocity $2C_p$ which implies that
$B(x(p,0)+C_Ne_{n+1},C_pt)$ will stand below $X_t$ for $t$ sufficient large. But
$M_t+C_Ne_{n+1}$ is above $X_t$ for all $t\geq 0$ and hence
$$B(x(p,0)+C_Ne_{n+1},C_pt)\cap M_t = \emptyset$$ for  $t$ sufficient large. Then we obtain a contradiction.

If $\alpha=2$, by (\ref{calculation1}) and (\ref{calculation3}), a direct calculation shows that
$0< H\leq CW$ at $t=0$. It follows from (\ref{preserved}) that
 $0< H\leq CW$ for $t\geq 0$. By the convexity, $|A|\leq n H\leq nCW\leq nC$.
 We use the contradictory arguments to get the lower bound for $|A|(p,t)$.
Assume that there exists $p\in M^n$ such that $|A|(p,t)\to 0$. We have for any small $\epsilon>0$ there exists $t_0>0$ 
such that $|A|(p,t)\leq \epsilon$ for $t\geq t_0$. Then by (\ref{mcf})
\begin{equation}\label{666}
 |x(p,t)-x(p,t_0)|\leq \epsilon t.
\end{equation}
Assume that $\lim\limits_{r\to\infty}\frac{u'(r)}{r}=\beta$. Taking $N=\beta$, there exists positive constant $C_1$ such that
$u(|y|)+C_1>g_1(|y|)$. Hence
$M_0+C_1e_{n+1}$ is above $X_0$.
Notice that $X_t$ is translating in the $e_{n+1}$ direction with velocity $\beta$ which implies that
$B(x(p,t_0)+C_1e_{n+1},\epsilon t)$ will stand below $X_t$ for $\epsilon<\beta$ and $t$ sufficient large. But
$M_t+C_1e_{n+1}$ is above $X_t$ for all $t\geq 0$ and hence
$$B(x(p,t_0)+C_1e_{n+1},\epsilon t)\cap (M_t+C_1e_{n+1}) = \emptyset$$ for $\epsilon<\beta$ and  $t$ sufficient large, which contradicts to (\ref{666}).

Finally we consider  the case $1<\alpha<2$. 
Since $M_t$ is  symmetric to $x_{n+1}$-axis and strictly convex for any $t\geq 0$,
for any $t$ there exists a point $q_t\in M_t$ achieves the unique minimum of the graph function for $M_t$. Moreover, $q_t$ always
stays on the $x_{n+1}$-axis, otherwise by the symmetry there would have more than one minimum point on $M_t$.  Denote $x(p_0,0)=q_0$ for some $p_0\in M^n$.
Since the all unit normal vectors at $q_t$ are $-e_{n+1}$,  we have $x(p_0,t)=q_t$ by (\ref{mcf}) for all $t\ge 0$.
Hence $\nu(p_0,t)=-e_{n+1}$ for any $t\geq 0$.
If $H(p_0,t)\geq c>0$, then by (\ref{mcf})
$$|x(p_0,t)-x(p_0,0)|\geq ct.$$
Since for any small $\epsilon>0$,  there exists positive constant $C_N$ such that
$ g_{\epsilon}(|y|)+C_N\geq u(|y|)$. Hence for any small $\epsilon>0$ we have
$X_0+C_Ne_{n+1}$ is above $M_0$.
Note that $X_t$ is translating in the $e_{n+1}$ direction with velocity $\epsilon$
and
$M_t$ is always below $ X_t+C_Ne_{n+1}$ for all $t\geq 0$ and hence
$$|x(p_0,t)-x(p_0,0)|\leq \epsilon t+C_N.$$
Then we obtain a contradiction when $\epsilon<c$ and $t$ is sufficient large. Hence $H(p_0,t)\to 0$ as $t\to+\infty$. It follows from $\frac{|\nabla{H}|}{H^2}\leq C_1(\delta)$ and the convexity
that $H(p,t)\to 0$ as $t\to+\infty$ for any $p\in M^n$. By the convexity, we conclude that $|A(p,t)|\to 0$ as $t\to+\infty$ for any $p\in M^n$.
$\Box$

Finally, we give the proof of Table 1.

\textbf{Proof of Table 1.}
Since
 the convex entire graph  $(y,u(|y|))$ with $u'(r)=O(r^{\alpha-1})$  over $\mathbb{R}^{n}$ for $0\leq\alpha\leq 1$  satisfies the condition (\ref{linear_growth}),  the mean curvature flow for such graph is Type III by Corollary 4.4 in \cite{EH}. Hence
 the normalized mean curvature flow (\ref{normalized_mcf})
converges as $s\to \infty$ to a self-expander
by Theorem \ref{limit} in the appendix. The rest of Table 1 follows from Theorem \ref{application_2} and Theorem \ref{main3}.
\end{proof}

\section{Asmptotic behavior of the Type IIb mean curvature flow}\label{sec_3}

In this section we study the asymptotic behavior of Type IIb mean curvature flow and prove the estimate (\ref{bound_3}) which is used in the proof of Theorem \ref{main1}. Huisken and Sinestrari (\cite{HS1} \cite{HS2}) proved blow-up  sequence of the compact mean curvature flow with positive mean curvature subconverges to a weakly convex  limit, which implies that
rescaled sequence (\ref{typeii}) of Type II compact mean curvature flow subconverges to a
translating soliton if $M_0$ is mean-convex. Haslhofer and Kleiner \cite{Ha1}
gave a short proof of Huisken and Sinestrari's result based on the noncollapsing theorem of the mean  curvature flow. Notice that Haslhofer and Kleiner's locally blow-up arguments in \cite{Ha1} also valid if noncollapsing theorem of the mean  curvature flow holds for noncompact hypersurface. So we only need to extend noncollapsing theorem of the mean  curvature flow to noncompact case, then we can show  Huisken and Sinestrari's result also holds for 
Type IIb case.

In \cite{A} Andrews gave a short quantitative argument about the result proved by Sheng and Wang \cite{SW1} that the compact mean-convex mean curvature flow satisfies $\delta$-Andrews' noncollapsing condition for all time for some $\delta>0$.
We extend Andrews' arguments to  noncompact case.
Precisely, we get
  the following theorem.
  \begin{thm}\label{BenAndrew}
  	Let $M_t$ be a solution to the mean curvature flow for mean-convex complete noncompact embedded hypersurface in $\mathbb{R}^{n+1}$ with bounded second fundamental form at each
  	time slice.
  	If $M_0$ satisfies the $\delta$-Andrews' noncollapsing condition, then it remains so under
  	the mean curvature flow.
  \end{thm}
In order to prove Theorem \ref{BenAndrew} we first need the following lemma.
\begin{lem}[Lemma 12.30 in \cite{Ricciflow}]\label{test_func}
Let $(M^n,g)$ be a complete noncompact Riemannian manifold with bounded sectional curvature $|sec|\leq k_0$ for some $k_0\geq 0$.
Then there exists constant $D=D(n,k_0)>0$ such that for any $O\in M^n$ there exists a $C^{\infty}$ function $h:M^n\to\mathbb{R}$ satisfying
\begin{equation*}
  D^{-1}(d_g(O,x)+1)\leq h(x)\leq D(d_g(O,x)+1)
\end{equation*}
and
\begin{equation*}
  |\nabla_g h|\leq D, \ \ \ \nabla_g\nabla_g h\leq D,
\end{equation*}
on $M^n$.
\end{lem}

Next we give the proof of Theorem \ref{BenAndrew}.

\textbf{Proof of Theorem \ref{BenAndrew}.}
We follow the Andrews' calculation in \cite{A}.
Denote
\begin{equation}
  Z(x,y,t)=\frac{H(x)}{2}||X(y,t)-X(x,t)||^2+\delta\langle X(y,t)-X(x,t),\nu(x)\rangle.
\end{equation}
We will prove that if $Z\geq 0$ at $t= 0$  then $Z\geq 0$ for all $t\geq 0$ for any constant $\delta$,
as showed in \cite{A}, that implies Theorem \ref{BenAndrew}.
Write $d=d(x,y,t)=||X(y,t)-X(x,t)||^2$, $\eta(x,y,t)=\frac{X(y,t)-X(x,t)}{d}$.
We choose the normal coordinates at $x$ and $y$.
By the equations (1) and (2) in \cite{A},
\begin{equation}\label{Z_derivative_1}
  \frac{\partial Z}{\partial x^i}=-dH_x\langle\eta,\partial^x_i\rangle+\frac{d^2}{2}\nabla_iH_x+\delta d h^x_{iq}g^{qp}_x\langle \eta,\partial^x_p\rangle,
\end{equation}
\begin{equation}\label{Z_derivative_2}
  \frac{\partial Z}{\partial y^i}=dH_x\langle\eta,\partial^y_i\rangle+\delta\langle \partial^y_i, \nu_x\rangle.
\end{equation}
Choose local coordinates so that $\{\partial_i^x\}$ are orthonormal and $\{\partial_i^y\}$ are orthonormal, and $\partial^x_i=\partial^y_i$
for $i=1,\cdots,n-1$. Thus $\partial^x_n$ and $\partial^y_n$ are coplanar with $\nu_x$ and $\nu_y$. Hence $\partial^x_n-\langle\partial^x_n,\partial^y_n \rangle\partial^y_n=
\langle\partial^x_n,\nu_y \rangle \nu_y$. By the calculation in \cite{A}
\begin{align}\label{1}
&  \frac{\partial Z}{\partial t}-\sum\limits_{i,j=1}^n(g_x^{ij}\frac{\partial^2 Z}{\partial x^i\partial x^j}+2g_x^{ik}g_y^{jl}\langle\partial^x_k,\partial^y_l\rangle\frac{\partial^2 Z}{\partial x^i\partial y^i}
+g^{ij}_y\frac{\partial^2 Z}{\partial y^i\partial y^i})\nonumber\\
 &=|A^x|^2Z+2d\langle\eta,\partial^x_i-\langle\partial^x_i,\partial^y_k\rangle g_y^{kl}\partial^y_l\rangle g_x^{ij}\nabla_j H_x-2(H_x-\delta h_{nn}^x)(1-\langle\partial^x_n,\partial^y_n\rangle^2)\nonumber\\
  &=|A^x|^2Z+2d\langle\eta,\partial^x_i-\langle\partial^x_i,\partial^y_k\rangle g_y^{kl}\partial^y_l\rangle g_x^{ij}\nabla_j H_x-2(H_x-\delta h_{nn}^x)\langle\partial^x_n,\nu_y \rangle^2
\end{align}
We get from (\ref{Z_derivative_1}) that
\begin{align}\label{2}
  &2d\langle\eta,\partial^x_i-\langle\partial^x_i,\partial^y_k\rangle g_y^{kl}\partial^y_l\rangle g_x^{ij}\nabla_j H_x\nonumber\\
  =& 2d\langle\eta,\partial^x_n-\langle\partial^x_n,\partial^y_n\rangle\partial^y_n\rangle(\frac{2}{d}\langle\eta,H_x\partial^x_n-\delta h^x_{nn}\partial^x_n\rangle
  +\frac{2}{d^2}\frac{\partial Z}{\partial x^n})\nonumber\\
 = &4(H_x-\delta h^x_{nn})\langle\eta,\nu_y\rangle\langle \partial^x_n,\nu_y\rangle\langle\eta,\partial^x_n\rangle
  +\frac{4}{d}\frac{\partial Z}{\partial x^n}\langle\eta,\nu_y\rangle\langle \partial^x_n,\nu_y\rangle.
\end{align}
Recall the Lemma 4 in \cite{A}
\begin{equation}\label{lem}
  \nu_y\sqrt{1+\frac{2H_x}{\delta^2}Z-\frac{1}{\delta^2}|\nabla_yZ|^2}=\nu_x+\frac{dH_x}{\delta}\eta-\frac{1}{\delta}\frac{\partial Z}{\partial y^q}g_y^{qp}\partial^y_{p},
\end{equation}
by writing $\rho'=\sqrt{1+\frac{2H_x}{\delta^2}Z-\frac{1}{\delta^2}|\nabla_yZ|^2}$, we have
\begin{equation}\label{lem1}
  \rho'\nu_y=\nu_x+\frac{dH_x}{\delta}\eta-\frac{1}{\delta}\frac{\partial Z}{\partial y^q}\partial^y_{q}.
\end{equation}
By (\ref{lem1}), we obtain that
\begin{align}\label{4}
  \langle \eta, \partial^x_n\rangle&=\frac{\delta}{dH_x}\langle \rho'\nu_y -\nu_x+\frac{1}{\delta}\frac{\partial Z}{\partial y^q}\partial^y_{q},\partial^x_n\rangle\nonumber\\
  &=\frac{\delta}{dH_x}\rho'\langle \nu_y,\partial^x_n\rangle+\frac{1}{dH_x}\frac{\partial Z}{\partial y^q}\langle\partial^y_{q},\partial^x_n\rangle.
\end{align}
It follows that
\begin{align}\label{11}
  \langle\eta,\nu_y\rangle\langle \partial^x_n,\nu_y\rangle\langle\eta,\partial^x_n\rangle=&\frac{\delta}{dH_x}\langle\eta,\rho'\nu_y\rangle\langle \partial^x_n,\nu_y\rangle^2+\frac{1}{dH_x}\frac{\partial Z}{\partial y^q}\langle\eta,\nu_y\rangle\langle \partial^x_n,\nu_y\rangle\langle\partial^y_{q},\partial^x_n\rangle\nonumber\\
  =&\frac{\delta}{dH_x}\langle\eta,\nu_x+\frac{dH_x}{\delta}\eta-\frac{1}{\delta}\frac{\partial Z}{\partial y^q}\partial^y_{q}\rangle\langle \partial^x_n,\nu_y\rangle^2\nonumber\\
  &+\frac{1}{dH_x}\frac{\partial Z}{\partial y^q}\langle\eta,\nu_y\rangle\langle \partial^x_n,\nu_y\rangle\langle\partial^y_{q},\partial^x_n\rangle\nonumber\\
   =&(\frac{Z}{d^2H_x}+\frac{1}{2})\langle\partial^x_n,\nu_y\rangle^2-\frac{1}{dH_x}\frac{\partial Z}{\partial y^q}\langle \eta,\partial^y_{q}\rangle\langle \partial^x_n,\nu_y\rangle^2\nonumber\\
  &+\frac{1}{dH_x}\frac{\partial Z}{\partial y^q}\langle\eta,\nu_y\rangle\langle \partial^x_n,\nu_y\rangle\langle\partial^y_{q},\partial^x_n\rangle,
\end{align}
where we use (\ref{lem1}) in the second equality and $Z=\frac{H_x}{2}d^2+\delta d\langle\eta,\nu_x\rangle$ in the last equality.
Combining with (\ref{1}), (\ref{2}), (\ref{4}) and (\ref{11}), we get

\begin{align}
&  \frac{\partial Z}{\partial t}-\sum\limits_{i,j=1}^n(g_x^{ij}\frac{\partial^2 Z}{\partial x^i\partial x^j}+2g_x^{ik}g_y^{jl}\langle\partial^x_k,\partial^y_l\rangle\frac{\partial^2 Z}{\partial x^i\partial y^i}
+g^{ij}_y\frac{\partial^2 Z}{\partial y^i\partial y^i})\nonumber\\
 &=(|A^x|^2+\frac{4(H_x-\delta h^x_{nn})}{d^2H_x}\langle \partial^x_n,\nu_y\rangle^2)Z
  +\frac{4}{d}\frac{\partial Z}{\partial x^i}\langle\eta,\nu_y\rangle\langle \partial^x_n,\nu_y\rangle\nonumber\\
  &\ \ -\frac{4(H_x-\delta h^x_{nn})}{dH_x}\frac{\partial Z}{\partial y^q}\langle \eta,\partial^y_{q}\rangle\langle \partial^x_n,\nu_y\rangle^2\nonumber\\
  &\ \ +\frac{4(H_x-\delta h^x_{nn})}{dH_x}\frac{\partial Z}{\partial y^q}\langle\eta,\nu_y\rangle\langle \partial^x_n,\nu_y\rangle\langle\partial^y_{q},\partial^x_n\rangle.
\end{align}
Assume that second fundamental form for $M_t$ is bounded by $C_0$ on $[0,T]$. Then
$|\nabla^m A|\leq C_m$ on $[0,T]$ by the standard estimates by Ecker and Huisken \cite{EH1}. By Lemma \ref{test_func} and direct
calculations, we get for $h$ which is the function defined in Lemma \ref{test_func}
\begin{equation*}
  D'^{-1}(d_{g(t)}(O,x)+1)\leq h(x)\leq D'(d_{g(t)}(O,x)+1)
\end{equation*}
\begin{equation*}
  |\nabla h|\leq D', \ \ \ \nabla \nabla h\leq D',
\end{equation*}
for $t\in[0,T]$, where $\nabla=\nabla_{g(t)}$ and $D'$ is a constant only depending on $D$, $C_0$ and $C_m$.
Let
$Q(x,y,t)=Z(x,y,t)+\epsilon\zeta(x,y,t)$ where $\zeta(x,y,t)=\xi(x,t)+\xi(y,t)$ with $\xi(x,t)=e^{(B+nD'+D'^2)t+h(x)}$  and
$\xi(y,t)=e^{(B+nD'+D'^2)t+h(y)}$, where  $B$ is positive constant to be determined later.
Then we have  
\begin{equation}\label{c1}
  (\frac{\partial}{\partial t}-\Delta)\xi \geq B \xi,
\end{equation}
\begin{equation}\label{c2}
 |\nabla \xi|\leq D'\xi,
\end{equation}
\begin{equation}
\xi(x,t) \geq e^{D'^{-1}(d_{g(t)}(O,x)+1)},
\end{equation}
on $M^n\times [0,T]$.

We claim that for all $\epsilon>0$ we have $Q(x,y,t)>0$ for all $x\neq y$ and $t\geq 0$. Assuming the claim and taking the
limit as $\epsilon\to 0$, we obtain Theorem \ref{BenAndrew}.
We prove the claim by contradiction. Notice that
\begin{align*}
\frac{Q}{d^2}=&\frac{H_x}{2}+\delta\frac{\langle X(y,t)-X(x,t),\nu_x\rangle}{d^2}+\epsilon\frac{\zeta}{d^2}\\
\geq& \frac{H_x}{2}+\delta\frac{\langle X(y,t)-X(x,t),\nu_x\rangle}{d^2}+\epsilon\frac{e^{D'^{-1}(d_{g(t)}(O,x)+1)}+e^{D'^{-1}(d_{g(t)}(O,y)+1)}}{d^2}.
\end{align*}
Since $\frac{H_x}{2}+\delta\frac{\langle X(y,t)-X(x,t),\nu(x)\rangle}{d^2}$ is uniformly bounded on $[0,T]$,
 for some $K_1$ sufficiently large, we have $Q>0$ when $d_{g(t)}(O,x)\geq K_1$ or $d_{g(t)}(O,y)\geq K_1$, and for some $k_1$ sufficient small, we have $Q>0$ when $d\leq  k_1$, where $K_1$ and $k_1$ is independent of $t$ and $B$.
 
 Now suppose that the claim is false. Then there exists a first time $t_0>0$, the points $x_0\neq y_0$ such that
 $Q(x_0,y_0,t_0)=0$ and $Q(x,y,t)>0$ for all $x, y\in M^n$ and $t<t_0$, moreover,
 \begin{equation}\label{dist}
  k_1\leq  d(x_0,y_0,t_0)\leq d_{g(t_0)}(O,x_0)+d_{g(t_0)}(O,y_0)\leq2K_1,
 \end{equation}
 where $k_1$ and $K_1$ are independent of $B$.
 Then at $(x_0,y_0,t_0)$, we have $\frac{\partial Q}{\partial x_i}=0$, $\frac{\partial Q}{\partial y_i}=0$.
It follows that at $(x_0,y_0,t_0)$
\begin{equation}\label{critical_point1}
-dH_x\langle\eta,\partial^x_i\rangle+\frac{d^2}{2}\nabla_iH_x+\delta d h^x_{iq}g^{qp}_x\langle \eta,\partial^x_p\rangle+\epsilon \frac{\partial \zeta}{\partial {x^i}}=0,
\end{equation}
and
\begin{equation}\label{critical_point2}
dH_x\langle\eta,\partial_{y^i}\rangle+\delta\langle \partial_{y^i}, \nu_x\rangle+\epsilon \frac{\partial \zeta}{\partial {y^i}}=0.
\end{equation}
At $(x_0,y_0,t_0)$ we have
\begin{align}\label{c3}
& 0\geq   \frac{\partial Q}{\partial t}-\sum\limits_{i,j=1}^n(g_x^{ij}\frac{\partial^2 Q}{\partial x^i\partial x^j}+2g_x^{ik}g_y^{jl}\langle\partial^x_k,\partial^y_l\rangle\frac{\partial^2 Q}{\partial x^i\partial y^i}
+g^{ij}_y\frac{\partial^2 Q}{\partial y^i\partial y^i})\nonumber\\
 &=\epsilon(\frac{\partial}{\partial t}-\Delta_x)\xi+\epsilon(\frac{\partial}{\partial t}-\Delta_y)\xi-\epsilon(|A^x|^2+\frac{4(H_x-\delta h^x_{nn})}{d^2H_x}\langle \partial^x_n,\nu_y\rangle^2)\zeta\nonumber\\
 &\ \
  +\epsilon\frac{4(H_x-\delta h^x_{nn})}{dH_x}\frac{\partial \xi}{\partial y^q}\langle \eta,\partial^y_{q}\rangle\langle \partial^x_n,\nu_y\rangle^2\nonumber\\
  &\ \ -\epsilon\frac{4(H_x-\delta h^x_{nn})}{dH_x}\frac{\partial \xi}{\partial y^q}\langle\eta,\nu_y\rangle\langle \partial^x_n,\nu_y\rangle\langle\partial^y_{q},\partial^x_n\rangle.
\end{align}
 Since the mean curvature is strictly positive on $[0,t_0]$ and by (\ref{dist}),
we have at $(x_0,y_0,t_0)$
\begin{equation}\label{c5}
  H_x\geq  k_2>0,
\end{equation}
where $k_2$ is a constant which is independent of $B$.
Notice that the second fundamental form is bounded by $C_0$ on $[0,T]$. Combining with (\ref{c1}), (\ref{c2}), (\ref{dist}),
(\ref{c3}) and (\ref{c5}), we get at $(x_0,y_0,t_0)$
\begin{equation}
  0\geq \epsilon(B-C')\zeta,
\end{equation}
where  $C'$ is a positive constant depending on $C_0$, $C_m$, $k_1$, $k_2$, $K_1$, $D'$ and dependent of $B$. Taking $B=C'+1$, we
obtain a contradiction.
$\Box$

Recall  Haslhofer and Kleiner \cite{Hasl} proved the following local estimate (See Theorem 1.8 in \cite{Hasl}).
\begin{thm}\label{nabla}\cite{Hasl}
 For any $\delta>0$ there exist $\rho(\delta)>0$ and $C_l(\delta)<+\infty$ such that if $M_t$ satisfies $\delta$-Andrews' noncollapsing condition in the parabolic ball $P(x,t,r)=B(x,r)\times (t-r^2,t]$ with $H(p,t)\leq r^{-1}$ then
 \begin{equation}
   \sup\limits_{P(x,t,\rho r)}|\nabla^lA|\leq C_l(\delta) r^{-(l+1)}.
 \end{equation}
\end{thm}

As the corollary of Theorem \ref{BenAndrew} and Theorem \ref{nabla}, we have the following
\begin{cor}\label{nabla1}
 Let $M_t$ be the mean curvature flow for mean-convex complete noncompact embedded hypersurface in $\mathbb{R}^{n+1}$ with bounded curvature at each time slice.
 If $M_0$ satisfies the $\delta$-Andrews' noncollapsing condition, then $\frac{|\nabla^l A|}{H^{l+1}}\leq C_l(\delta)$ for any $t\in (0,T)$.
\end{cor}
\begin{proof}
  By Theorem \ref{BenAndrew}, $M_t$ satisfies the $\delta$-Andrews' noncollapsing condition for all $t\geq 0$. Then Corollary \ref{nabla1} follows from Theorem \ref{nabla} directly.
\end{proof}

Also recall  Haslhofer and Kleiner \cite{Hasl} proved (see Corollary 2.15 in \cite{Hasl})
\begin{thm}\label{nabla22}\cite{Hasl}
  If $M_t$ is an ancient mean-convex smooth  mean curvature flow satisfies the $\delta$-Andrews' noncollapsing condition, then
  $M_t$ is weakly convex.
\end{thm}

Finally as the application of Theorem \ref{BenAndrew} and Theorem \ref{nabla22}, we give
a proof of Theorem \ref{andrew_application2}.

\textbf{Proof of Theorem \ref{andrew_application2}.}
  By Theorem \ref{BenAndrew}, $M_t$ satisfies the $\delta$-Andrews' noncollapsing condition for all $t\geq 0$. By taking $y\to x$ in the $\delta$-Andrews' noncollapsing condition,
  we have $-Hg_{ij}\leq \delta h_{ij}\leq Hg_{ij}$ for all $t\geq 0$. Then the second fundamental forms are uniformly bounded for the recaled sequence (\ref{TypeIIb_rescaling}). Hence
  the limit of the recaled sequence (\ref{TypeIIb_rescaling}) is a mean-convex eternal solution $M_{\infty}$ satisfying $\delta$-Andrews' noncollapsing condition. It follows from
  Theorem \ref{nabla22} that the $M_{\infty}$ is weakly convex. Then Corollary \ref{andrew_application2} follows from the strong maximum principle and  Hamilton's Harnack inequality (See Main Theorem B in \cite{Hamilton}).
$\Box$

\section{appendix }
In this section we prove that the limit of rescaled sequence (\ref{scaling}) for convex Type III mean curvature flow is self-expander (Corollary
\ref{self_expander}).
Similar results had been obtained by Hamilton \cite{Ha1} \cite{Hamilton} for Type II Ricci flow and mean curvature flow and Chen and Zhu \cite{CZ} for
Type III Ricci flow. One can use the similar arguments to prove Corollary
\ref{self_expander}.
We give a proof for sake of convenience for the readers.

First we recall Hamilton's Harnack inequality for the mean curvature flow.
\begin{thm}\cite{Hamilton}
For any weak convex solution to mean curvature flow for $t>0$ we have
\begin{equation}\label{Harnack}
 \tilde{ Z}=\frac{\partial H}{\partial t}+\frac{H}{2t}+2V_i\nabla_i H+h_{ij}V_iV_j\geq 0,
\end{equation}
for all tangent vectors $V$.
\end{thm}
\begin{thm}\label{limit}
Any strictly convex  solution to the mean
curvature flow where $\frac{\partial}{\partial t}(\sqrt{t}H)=0$ at some point $(x_0,t_0)$ for $t_0>0$   must be the self-expander.
\end{thm}
\begin{proof}
Recall Hamilton proved (see Corollary 4.4 in \cite{Hamilton}) the Harnack quantity $\tilde{Z}$ satisfies
\begin{equation}
  (D_t-\Delta)\tilde{Z}=(|A|^2-\frac{2}{t})\tilde{Z}+2\tilde{X}_a\tilde{U}_a
  -2h_{bc}\tilde{Y}_{ab}\tilde{Y}_{ac}-4\tilde{Y}_{ab}\tilde{W}_{ab},
\end{equation}
with
$\tilde{X}_a=\nabla_aH+h_{ab}V_b$, $\tilde{Y}_{ab}=\nabla_aV_b-Hh_{ab}-\frac{1}{2t}g_{ab}$,
$\tilde{W}_{ab}=\frac{\partial}{\partial t}h_{ab}+V_c\nabla_ch_{ab}+\frac{1}{2t}h_{ab}$,
$\tilde{U}_a=(\frac{\partial}{\partial t}-\Delta)V_a+h_{ab}\nabla_b H+\frac{1}{t}V_a$.
Since $\frac{\partial}{\partial t}(\sqrt{t}H)=0$ at some point $(x_0,t_0)$ for $t_0>0$ , we know that
at this point
\begin{equation}\label{yoga1}
\frac{\partial H}{\partial t}+\frac{H}{2t}=0.
\end{equation}
Taking $V_i=-h^{-1}_{ij}\nabla_j H$ in (\ref{Harnack}), we
have at $(x_0,t_0)$
$$ -h^{-1}_{ij}\nabla_i H\nabla_j H\geq 0.$$
It follows that at $(x_0,t_0)$
\begin{equation}\label{yoga2}
 \nabla H=0.
 \end{equation}
Then we obtain that $\tilde{Z}=0$ in the $V=0$ direction. The strong maximum principle implies that
there exists vector $V$ at each point such that $\tilde{Z}=0$. Moreover, the zero factor $V$ is obtained from the first variation of $\tilde{Z}$ by
$V_a=-h_{ab}^{-1}\nabla_bH$.

Now fix $V_a=-h_{ab}^{-1}\nabla_bH$ at $(x_0,t_0)$ and extend $V$ in a neighborhood of $(x_0,t_0)$ in space-time such that
$$
\tilde{U}_a=(\frac{\partial}{\partial t}-\Delta)V_a+h_{ab}\nabla_b H+\frac{1}{t}V_a=\tilde{X}_a,
$$
and
$$
\tilde{Y}_{ab}=\nabla_aV_b-Hh_{ab}-\frac{1}{2t}g_{ab}=-\tilde{W}_{ad}h^{-1}_{db}.
$$
 Then at $(x_0,t_0)$
\begin{align*}
  0&\geq (\frac{\partial}{\partial t}-\Delta)\tilde{Z}\\
  &=(|A|^2-\frac{2}{t})\tilde{Z}-4\tilde{W}_{ab}\tilde{Y}_{ab}+2\tilde{X}_a\tilde{U}_a-2h_{ac}\tilde{Y}_{bc}\tilde{Y}_{ba}\\
  &=4\tilde{W}_{ad}h^{-1}_{db}\tilde{W}_{ba}+2|\tilde{X}_a|^2-2h_{ac}\tilde{W}_{bd}h_{dc}^{-1}\tilde{W}_{be}h^{-1}_{ea}\\
  &=2\tilde{W}_{ad}h^{-1}_{db}\tilde{W}_{ba}+2|\tilde{X}_a|^2,
\end{align*}
which implies that
$$\tilde{W}_{ab}=0,\ \ \tilde{X}_a=0.$$
Thus we obtain
\begin{equation}\label{appendix_1}
\frac{\partial}{\partial t}h_{ab}+V_c\nabla_ch_{ab}+\frac{1}{2t}h_{ab}=0,\ \  \nabla_aH+h_{ab}V_b=0,
\end{equation}
for $V_a=-h_{ab}^{-1}\nabla_b H$ everywhere.
We get from differentiating the second equation in (\ref{appendix_1}) that
\begin{equation}\label{appendix_2}
\nabla_a\nabla_b H+V_c\nabla_a h_{bc}+h_{bc}\nabla_aV_c=0.
\end{equation}
It follows from Theorem 2.3 in \cite{Ha1} that
\begin{align}\label{appendix_3}
 &\nabla_a\nabla_b H
 =\nabla_a\nabla_ch_{bc}\nonumber\\
 =&\nabla_{c}\nabla_ah_{bc}+R_{acbd}h_{dc}+R_{accd}h_{bd}\nonumber\\
 =&\Delta h_{ab}+(h_{ab}h_{cd}-h_{ad}h_{bc})h_{dc}+(h_{ac}h_{cd}-h_{ad}H)h_{bd}\nonumber\\
 =&\frac{\partial }{\partial t}h_{ab}-Hh_{ad}h_{bd}.
\end{align}
By (\ref{appendix_1}), (\ref{appendix_2}) and (\ref{appendix_3}), we get
\begin{equation}\label{appendix_4}
  \nabla_aV_c=Hh_{ac}+\frac{1}{2t}g_{ac}.
\end{equation}
Consider the vector
$$\frac{1}{2t}T^{\alpha}=g^{ij}V_i\nabla_jX^{\alpha}+H\nu^{\alpha}+\frac{1}{2t}X^{\alpha},$$
where $\nu=(\nu^1,\cdots,\nu^{n+1})$ is the unit normal vector of $X$.
By (\ref{appendix_1}) and (\ref{appendix_4}),
we have
$$\nabla_k T^{\alpha}=g^{ij}\nabla_kV_i\nabla_jX^{\alpha}+g^{ij}V_ih_{jk}\nu^a+(\nabla_kH)\nu^{\alpha}-h_{kj}g^{jm}\nabla_mX^{\alpha}+\frac{1}{2t}\nabla_kX^{\alpha}=0.$$
Then
$$g^{ij}V_i\nabla_jX^a+H\nu^{\alpha}+\frac{1}{2t}(X^{\alpha}-T^{\alpha})=0.$$
It follows that $T$ is a constant vector.
Taking the vertical part,
we have
$$H\nu^{\alpha}+\frac{1}{2t}(X^{\alpha}-T^{\alpha})^{\perp}=0.$$
\end{proof}

Finally, we give the proof of Corollary \ref{self_expander}.
\begin{cor}\label{self_expander}
  Let $M_t$ be the Type III convex mean curvature flow for the noncompact hypersurface with bounded second
  fundamental form at each time slice.
  Then the limit obtained as (\ref{scaling}) is a non-flat self-expander splitting as $\mathbb{R}^{n-k}\times \Sigma^k$, where $\Sigma^k$ is strictly convex.
\end{cor}
\begin{rem}
Due to a counter-example in \cite{CN}(see Example 3.4 in \cite{CN}), Corollary \ref{self_expander} is not true if we only assume the Type III mean curvature flow is mean-convex.
\end{rem}
\begin{proof}
  By Hamilton's Harnack (\ref{Harnack}), we have $\sqrt{t}H$ is pointwisely monotone nonincreasing. Taking $\tilde{x}_i(s)=\tilde{x}(s+s_i)$, where $\tilde{x}$ is defined in
  (\ref{scaling}) and $s_i\to+\infty$. We take the limit as the way in the proof of Theorem \ref{main1}. Let $p\in M^n$ be the based point taken in the proof of Theorem \ref{main1}.
  Then  $\tilde{H}_{\infty}(p_{\infty},s)=\lim\limits_{i\to\infty}\tilde{H}(p,s+s_i)=\lim\limits_{i\to\infty}\sqrt{2(t+t_i)+1}H(p,t+t_i)=
  \sqrt{2}\lim\limits_{i\to\infty}\sqrt{t+t_i}H(p,t+t_i)\equiv constant>0$, where  $s_i=\frac{1}{2}\log(2t_i+1)$.
  Then strong maximum principle we know the limit splitting as $\mathbb{R}^{n-k}\times \Sigma^k$, where $\Sigma^k$ is strictly convex.
  Hence
   Corollary \ref{self_expander} holds by Theorem \ref{limit}.
\end{proof}

\end{document}